\documentclass[12pt]{article}
\usepackage[top=2.5cm,bottom=2.5cm,left=2.5cm,right=2.5cm]{geometry}

\usepackage{mathrsfs}
\usepackage{pstricks}
\usepackage{amsmath,amssymb,graphicx,amsthm,tabularx,array}
\usepackage{cite}
\usepackage{hyperref}

\theoremstyle{plain}
\newtheorem{theorem}{Theorem}[section]

\newtheorem{coro}[theorem]{Corollary}
\newtheorem{lemma}[theorem]{Lemma}
\newtheorem{defn}[theorem]{Definition}

\theoremstyle{definition}

\newtheorem{other}{}

\DeclareMathOperator{\dist}{dist}
\DeclareMathOperator{\diam}{diam}

\title{On the first and second eigenvalue of finite and infinite uniform hypergraphs}

\author{Hong-Hai Li\thanks{Supported by National Natural Science Foundation of China
(11561032, 11201198) and CSC, Natural Science Foundation of
Jiangxi Province (20142BAB211013), the Sponsored Program for Cultivating Youths of
Outstanding Ability in Jiangxi Normal University. 
The work was done while this author visited the Simon Fraser University.}\\
College of Mathematics and Information Science\\
Jiangxi Normal University,\\
Nanchang, Jiangxi 330022,  China\\
{\tt lhh@jxnu.edu.cn}
\and
Bojan Mohar\thanks{Supported in part by an NSERC Discovery Grant (Canada), by the Canada Research Chairs program, and by the Research Grant P1--0297 of ARRS (Slovenia).}\\
{Department of Mathematics}\\
{Simon Fraser University}\\
{Burnaby, B.C. V5A 1S6} \\
{\tt mohar@sfu.ca}
}

\begin{document}

\maketitle

\begin{abstract}
Lower bounds for the first and the second eigenvalue of uniform hypergraphs which are regular and linear are obtained.
One of these bounds is a generalization of the Alon-Boppana Theorem to hypergraphs.
\end{abstract}



\section{Introduction}

A \emph{hypergraph} $G$ is a pair $(V,E)$, where $E\subseteq \mathcal{P}(V)$ and $\mathcal{P}(V)$ stands for the power set of $V$. The elements of $V=V(G)$, are referred to as \emph{vertices} and the elements of $E=E(G)$ are called \emph{edges}.
A hypergraph $G$ is \emph{$t$-uniform} for an integer $t\geq2$ if every edge $e\in E(G)$ contains precisely $t$ vertices, $|e|=t$.  For a vertex $v\in V$, we denote by $E_v$ the set of edges containing $v$, i.e., $E_v = \{e\in E \mid v\in e\}$. The cardinality $|E_v|$ is the \emph{degree} of $v$. A hypergraph is \emph{regular of degree $k$} if all of its vertices have degree $k$. If any two edges in $H$ share at most one vertex, then $H$ is said to be a \emph{linear hypergraph}.  As implicitly indicated by the definition, where the edges are subsets of the vertex-set, we assume in this paper that hypergraphs are \emph{simple} (i.e., there are no multiple edges). We will also consider (countably) infinite hypergraphs, but only those that are \emph{locally finite}, i.e., each vertex is of finite degree.

A sequence of edges $e_1,\ldots,e_l$ such that $e_i\cap
e_{i+1} \neq \emptyset$ for $i=1,\dots,l-1$ is referred to as a \emph{walk}. The
\emph{length} of a walk is the number of edges in it. We say that a walk
$e_1,\ldots,e_l$ \emph{connects} vertices $u,v\in V$ if $u\in e_1$
and $v\in e_l$. The hypergraph is \emph{connected} if for
each pair of vertices $u,v\in V$, there exists a walk connecting
them. The \emph{distance} between two vertices $u$ and $v$ is
the minimum length of a walk connecting $u$ and
$v$ in $G$, and it is denoted by $\dist(u,v)$. The \emph{diameter} of
$G$, denoted by $\diam(G)$, is the maximum distance between a pair of vertices
in $G$.

To each hypergraph $H$ we associate its \emph{Levi graph} $L(H)$ whose vertices are
$V(H)\cup E(H)$, and each $e\in E(H)$ is adjacent to all vertices $v\in e$.
The hypergraph is \emph{acyclic} if its Levi graph is a forest. Note that, in particular,
this means that $H$ is linear.

In this paper we discuss the largest and the second largest eigenvalue of
finite and infinite hypergraphs. In the case of graphs, it is well-known that the difference between
these two eigenvalues, the \emph{spectral gap}, is intimately
related to expansion properties of the graph. Roughly speaking,
the larger the gap, the closer $G$ is to resemble a random
graph, and the more expanding it is (see, e.g., \cite{KrSh11}).

For $k$-regular graphs of fixed degree $k$, the spectral gap is larger when the
second eigenvalue $\lambda_2(G)$ is smaller. Thus a question arises, how small can $\lambda_2(G)$ be?
The question was answered by the well-known
Alon-Boppana Theorem (see Nilli~\cite{Nilli_dm_1991} or \cite{Mo_ABTheorem}),
which is stated below. Let us observe that there is a strong connection between the second eigenvalue of a $k$-regular
graph and the spectral radius of the infinite $k$-regular tree,
its universal cover, whose spectrum is the interval
$[-2\sqrt{k-1},2\sqrt{k-1}]$ with $2\sqrt{k-1}$ being its spectral
radius.

\begin{theorem}[Alon-Boppana]
\label{thm:Alon-Boppana graphs}
 For every $k$-regular graph of order $n$, its second largest eigenvalue $\lambda_2$ satisfies
\[
\lambda_2\geq 2\sqrt{k-1}-o_n(1)
\]
where $o_n(1)$ is a quantity that tends to zero for every fixed $k$ as $n\rightarrow\infty$.
\end{theorem}

Many authors have attempted to try and develop a parallel theory
of expansion  that applies to hypergraphs, for it has
applications in other fields such as  parallel and distributed
computing (see, e.g., \cite{cataayka_IEEE_1999}). Friedman and
Wigderson~\cite{FriedWigd_combi_95}, and
Chung~\cite{chung_dimacs_93} each proposed an operator attached to
a hypergraph and defined the second largest eigenvalue of
hypergraphs differently. However, they failed to show the
threshold bound for their operators, namely, the bound analogous
to the $2\sqrt{k-1}$ bound for ordinary $k$-regular graphs. For
this, Feng and Li~\cite{FengLi_numbertheory_96} used another
matrix to extend the analogous threshold bound to regular
hypergraphs.  However, since then, no further significant
improvements have been discovered. To study hypergraph
quasirandomness, Lenz and Mubayi~\cite{LenzMuba} adopted the
definitions of the first and second eigenvalues of hypergraphs in
\cite{FriedWigd_combi_95} and showed that, for uniform
hypergraphs, a property \texttt{Eig} involving eigenvalues is
equivalent to property \texttt{Disc}, which states that all
sufficiently large vertex sets have roughly the same edge density
as the entire hypergraph.

The extremal (largest and smallest) eigenvalues of hypergraphs have
also been studied in a recent
enlightening paper of Nikiforov \cite{Nikiforov14}. In this seminal work,
it is noted that in case of hypergraphs, it is usually more natural
to consider more general $L^p$-norms instead of the $L^2$-norms.
The outcomes of our paper confirm that this principle extends to the notion
of the second eigenvalue as well.

In this paper we use the definition essentially identical to that
of Friedman and Wigderson \cite{FriedWigd_combi_95}, but working with the $L^t(V)$ norm for
$t$-uniform hypergraphs rather than $L^2(V)$ norm. Interestingly,
unlike the later work~\cite{FriedWigd_combi_95}, Friedman's early
paper~\cite{Friedman_siamcomp_1991} did this indeed. In
\cite{Friedman_siamcomp_1991}, the notion of the spectrum
of infinite hypergraphs was introduced. A model of the infinite
hypertree, the $t$-uniform $k$-regular hypertree
$T_{t,k}$ was considered and a precise value for
the spectral radius was obtained
\cite[Proposition 3.2]{Friedman_siamcomp_1991}. The detailed proof
is given for the 3-uniform hypertree and the extension to the
$t$-uniform case is indicated at the end of the paper.
Let us recall that the \emph{$t$-uniform $k$-regular hypertree} $T_{t,k}$ is a connected infinite
$t$-uniform acyclic hypergraph in which each vertex has degree $k$.
In other words, its Levi graph is isomorphic to the $(t,k)$-biregular infinite tree.

\begin{theorem}[Friedman \cite{Friedman_siamcomp_1991}] \label{thm_Tktradius}
The spectral radius of\/ $T_{t,k}$ (in its $L^t$ norm) is
$$(k-1)^{1/t}t!(t-1)^{(1-t)/t}.$$
\end{theorem}

Naturally, this motivates us to think of it analogously as in Theorem \ref{thm:Alon-Boppana graphs} to state
a threshold bound for the second eigenvalue of regular
hypergraphs.  Indeed, it is shown in Section \ref{sect:main results} that there is an
exact analogy to the graph case. We use it first to set a lower
bound for the first eigenvalue of infinite regular hypergraphs in
terms of this threshold bound (see Theorem \ref{thm_radius}). A
direct corollary of this bound is that the spectral radius of any
(finite or infinite) acyclic hypergraph with maximum degree at
most $k$ is at most $\frac{t}{t-1}((t-1)(k-1))^{1/t}$, see
Corollary \ref{cor:acyclic}.

In addition, we set up a version of the Alon-Boppana Theorem for finite
hypergraphs, and indicate the agreement with the graph case
(see Theorems \ref{thm_AlonBoppana hypergraphs} and \ref{thm_Serre hypergraphs}).

\section{Eigenvalues of hypergraphs}

In this section, we give the definitions of the first and second eigenvalues of a hypergraph.
These definitions are essentially identical to those
given in \cite{FriedWigd_combi_95} but with $L^t(V)$ norm for $t$-uniform hypergraphs,
and our concepts fit well with the algebraic definition of eigenvalues of tensors
proposed independently by Qi~\cite{qi05} and Lim\cite{Lim05}.
The $t$-spectral radius through the variation of the multilinear form for $t$-uniform hypergraphs was also studied in \cite{KeevLenzMubayi14,KangNikiYuan2015,Nikiforov14}.

\begin{defn}
Let $W_1, W_2,\ldots, W_t$ be finite-dimensional vector spaces over $\mathbb{C}$, and let $\phi: W_1\times\cdots\times W_t\rightarrow \mathbb{C}$ be a $t$-linear map. The \emph{spectral norm} of $\phi$ is
\[
    \Vert\phi\Vert = \sup \{|\phi(x_1,\ldots,x_t)| : x_i\in W_i,~\|x_i\|_t=1,~i=1,\dots,t\}.
\]
\end{defn}

If $\phi$ is symmetric (and $W=W_1=\cdots=W_t$), then it is easy to see that $\Vert\phi\Vert$ is attained with $x_1=\cdots=x_t$, i.e.
$$\Vert\phi\Vert = \sup \{|\phi(x,\ldots,x)| : x\in W,~\|x\|_t=1\}.$$

\begin{defn}\label{defn_radius_Fried}
Let $H$ be a $t$-uniform hypergraph. The \emph{adjacency map} of $H$ is the symmetric $t$-linear map $\tau_H : W^t\rightarrow \mathbb{C}$ defined as follows, where $W = \mathbb{C}^{V(H)}$. For all $v_1,\ldots,v_t\in V(H)$
\begin{equation}
\label{eq:tauH}
   \tau_H(e_{v_1},\ldots,e_{v_t})=\left\{
                                 \begin{array}{cl}
                                   \frac{1}{(t-1)!}, & \hbox{if $\{v_1,\ldots,v_t\}\in E(H)$;} \\[1mm]
                                   0, & \hbox{otherwise,}
                                 \end{array}
                               \right.
\end{equation}
where $e_v$ denotes the indicator vector of the vertex $v$, that is the vector which has $1$ in coordinate $v$ and $0$ in all other coordinates.  The adjacency map $\tau_H$ is then extended by linearity to $W^t$.

Similarly, we define the all-ones map $J : W^t\rightarrow \mathbb{C}$ to be $J(e_{v_1},\ldots,e_{v_t})=1$ for the standard basis vectors $e_{v_1},\ldots,e_{v_t}$ of $W$, and then extend it by linearity to all of the domain.

The \emph{largest eigenvalue} (or \emph{spectral radius}) of $H$ is defined to be $\rho(H)=\Vert\tau_H\Vert$,
and the \emph{second largest eigenvalue} of $H$ is $\lambda_2(H) = \Vert\tau_H-\frac{t|E(H)|}{n^t}J\Vert$.
\end{defn}

Let us recall the tensor algebra setup introduced by Qi~\cite{qi05}, which can be used in spectral hypergraph theory.

An \emph{$n$-dimensional real tensor} $\mathcal{T}$ of \emph{order} $t$ consists of
$n^t$ entries of real numbers:
\[
\mathcal{T}=(\mathcal{T}_{i_1i_2\cdots
i_t}),\,\,\mathcal{T}_{i_1i_2\cdots i_t}\in \mathbb{R},\qquad i_1,i_2,\ldots,i_t\in \{1,\dots,n\}.
\]
For tensors corresponding to hypergraphs, we shall adopt the common setup from graph theory, where the coordinates correspond to vertices of the (hyper)graph $H$, and thus the indices $i_1,i_2,\ldots,i_t\in \{1,\dots,n\}$ are replaced by vertices,
$v_1,v_2,\ldots,v_t\in V(H)$, and $n=|V(H)|$. Then we say that $\mathcal T$ is a \emph{tensor of order $t$ over} $V=V(H)$.
A tensor $\mathcal{T}$ over $V$ is \emph{symmetric} if the value of $\mathcal{T}_{v_1v_2\cdots v_t}$ is invariant under any
permutation of the indices $v_1,\ldots,v_t$.

By considering $x\in \mathbb{C}^V$ as a tensor of order 1, we can consider the product $\mathcal{T}x$, which is a vector in $\mathbb{C}^V$, whose $v$-component is
\begin{equation}\label{e-produdefn}
(\mathcal{T}x)_v = \sum_{v_2,\ldots,v_t\in V} \mathcal{T}_{v v_2 \cdots v_t}\, x_{v_2}\cdots x_{v_t}.
\end{equation}

A tensor $\mathcal{T}$ of order $t$ over $V$ defines a homogeneous polynomial of degree $t$ with variables $x_v$ ($v\in V$) given by the formula
\begin{equation}\label{e-multilinear form}
x^T(\mathcal{T}x) = \sum_{v_1,\ldots,v_t\in V} \mathcal{T}_{v_1 \cdots v_t}\, x_{v_1}\cdots x_{v_t}.
\end{equation}

These definitions extend to the case when $V$ is infinite, and we shall use these only for the case when the sum in (\ref{e-produdefn}) is finite for each $v$ (the hypergraph $H$ is locally finite) and $x$ is such that the form (\ref{e-multilinear form}) is convergent.

\begin{defn}
Let $\mathcal{T}$ be a tensor of order $t$ over $V$. Then $\lambda$ is an \emph{eigenvalue} of
$\mathcal{T}$ and $0\neq x\in \mathbb{C}^V$ is an \emph{eigenvector} corresponding to $\lambda$ if
$(\lambda,x)$ satisfies
\begin{equation*}
\mathcal{T}x = \lambda x^{[t-1]},
\end{equation*}
where $x^{[t-1]}\in \mathbb{C}^V$ is defined as $(x^{[t-1]})_v = (x_v)^{t-1}$.
\end{defn}



Let $\mathcal{T}$ be a nonnegative tensor. If $\mathcal T$ is finite-dimensional, then the \emph{spectral radius} of $\mathcal{T}$ is defined as $\rho(\mathcal{T})=\max\{|\lambda|\,:\, \lambda \hbox{ is an eigenvalue of } \mathcal{T}\}$.
If $\mathcal T$ is infinite-dimensional, then we define the \emph{spectral radius} of $\mathcal T$ as
$$
   \rho(\mathcal T) = \sup \{\rho(\mathcal T_U) \mid U\subset V,~U\hbox{ finite}\},
$$
where $\mathcal T_U$ is the tensor over $U$ which coincides with $\mathcal T$ for every $t$-tuple of elements from $U$.

For $x\in \mathbb{C}^V$, we define the \emph{$t$-norm} $\Vert x\Vert_t$ of $x$ by
$$
   \Vert x\Vert_t = \biggl(\, \sum_{v\in V} |x_v|^t \biggr)^{1/t}.
$$

Let $\mathbb{R}_{+}^V = \{x\in \mathbb{R}^V \ | \ x\ge 0\}$. The following lemma is an analogue of the basic form of the Perron-Frobenius Theorem for nonnegative matrices. It was proved for the finite-dimensional case in \cite{HuQi13normLap}, and it extends easily to the infinite-dimensional case.

\begin{lemma}\label{lem-radinonnegsym}
Let $\mathcal{T}$ be a symmetric nonnegative tensor of order $t$ over $V$. Then
\begin{equation}\label{eradiusnonsymten}
  \rho(\mathcal{T})=\sup\{x^T(\mathcal{T}x) \mid x\in \mathbb{R}_{+}^V, \Vert x\Vert_t = 1\}.
\end{equation}
If $\mathcal T$ is finite-dimensional, then the supremum in $(${\rm \ref{eradiusnonsymten}}$)$ is attained, and $x\in \mathbb{R}_{+}^V$ with $\Vert x\Vert_t = 1$ is an eigenvector of $\mathcal{T}$ corresponding to $\rho(\mathcal{T})$ if and only if it is an optimal solution of the maximization problem \eqref{eradiusnonsymten}.
\end{lemma}

We refer to Nikiforov \cite{Nikiforov14} for overview of combinatorial results using the spectral radius of hypergraphs.

For a $t$-uniform hypergraph $H = (V,E)$, we define its \emph{adjacency tensor} as the tensor $\mathcal A$ of order $t$ over $V$, whose $(v_1, \ldots,v_t)$-entry is:
\begin{eqnarray*}
{\mathcal A}_{v_1 \ldots v_t} = \tau_H(e_{v_1},\ldots,e_{v_t}) = \left\{\begin{array}{cl}\frac{1}{(t-1)!} &\mbox{if}\;\{v_1,\ldots,v_t\}\in E,\\[0.5mm]
                      0 & \mbox{otherwise}.\end{array}\right.
\end{eqnarray*}

Let us observe that the adjacency tensor defined above differs from the multilinear
form associated with hypergraphs in \cite{Friedman_siamcomp_1991}
by the scalar factor $\frac{1}{(t-1)!}$. In this sense, Theorem~\ref{thm_Tktradius} gives the spectral radius
of the infinite regular hypertree in terms of its adjacency tensor introduced above.

\begin{theorem}[Friedman \cite{Friedman_siamcomp_1991}] \label{thm_Tktradius as tensor}
The spectral radius of the adjacency tensor of the infinite $k$-regular hypertree\/ $T_{t,k}$ is
$$\rho(T_{t,k}) = \frac{t}{t-1}((t-1)(k-1))^{1/t}.$$
\end{theorem}

For a vector $x \in \mathbb{C}^V$ and a subset $U\subseteq V$, we write
$$x^U=\prod_{u\in U}x_u\,.$$

By \cite{CoopDut12}, we have
$$x^T(\mathcal{A}(H)x)=\sum_{e\in E(H)} t\, x^e$$
and
$$(\mathcal{A}(H)x)_v = \sum_{e\in E_v} x^{e\setminus \{v\}}.$$

It is easy to see that the spectral radius of a finite hypergraph
$H$ given by Definition~\ref{defn_radius_Fried} agrees with
the spectral radius of the adjacency tensor of $H$.

Several other attempts have been made to define the eigenvalues of hypergraphs such as a definition by Lu and Peng~\cite{LuPeng_algorithmsmodels_2011,LuPeng_randomstructurealgorithms_2012} and Chung~\cite{chung_dimacs_93} using matrices,  the eigenvalues of the shadow graph~\cite{Mart_procams_01} and so on.

\section{Main results}
\label{sect:main results}

In this section, we first prove that the spectral radius of any
(finite or infinite) regular uniform hypergraph is bounded below
by that of the infinite hypertree. Next, we generalize the
Alon-Boppana Theorem to hypergraphs, obtaining the threshold bound
for the second eigenvalue of regular uniform hypergraphs.

Now we introduce a function $g: \mathbb{N}\rightarrow\mathbb{R}$ by setting
\[
g(n)=\frac{1+((t(1-\frac{1}{k}))^{1/(t-1)}-1)n}{((t-1)(k-1))^{n/t}},
\]
where $t\geq2$, $k\geq2$ are fixed integers. Note that $g(0)=1$. For convenience, let
$$\hat{g}(n) = 1+((t(1-\frac{1}{k}))^{1/(t-1)}-1)n = g(n)((t-1)(k-1))^{n/t}.$$
First, we show that $g$ is a non-increasing function.

\begin{lemma}\label{lem_decrease}
Let $t\geq2$ and $k\geq2$ be integers. Then for every $n\ge0$, $g(n+1)\le g(n)$.
\end{lemma}

\begin{proof}
If $t=2$ and $k=2$, then $g(n)=1$ for every $n$. The inequality is also clear for $n=0$. Thus, we may assume that $n\ge1$ and that $t>2$ or $k>2$. Since
\begin{align*}
    \frac{g(n+1)}{g(n)}=((t-1)(k-1))^{-1/t}\left(1+\frac{(t(1-\frac{1}{k}))^{1/(t-1)}-1}{1+((t(1-\frac{1}{k}))^{1/(t-1)}-1)n}\right),
\end{align*}
it suffices to prove that
$$
   1+\frac{(t(1-\frac{1}{k}))^{1/(t-1)}-1}{1+((t(1-\frac{1}{k}))^{1/(t-1)}-1)n}
   \leq((t-1)(k-1))^{1/t}
$$
for every $n\ge1$. The right-hand side is independent of $n$ and the left-hand side is largest when $n=1$. Thus, it suffices to prove that
$$
  1+\frac{(t(1-\frac{1}{k}))^{1/(t-1)}-1}{(t(1-\frac{1}{k}))^{1/(t-1)}}
  \leq ((t-1)(k-1))^{1/t}.
$$
This is equivalent to the inequality
$((t-1)(k-1))^{1/t}+(t(1-\frac{1}{k}))^{1/(1-t)}\geq2$,
which we prove in the remainder of the proof.
Equivalently, we shall prove that
$(sm)^{1/(s+1)}+(\frac{m+1}{m(s+1)})^{1/s}\geq2$, for $s\geq1$ and
$m\geq1$. By AM-GM inequality, we only need to show that
$(sm)^{1/(s+1)}(\frac{m+1}{m(s+1)})^{1/s}\geq1$. Since
\begin{align*}
   (sm)^{s}\left(\frac{m+1}{m(s+1)}\right)^{s+1}&= \frac{s^s}{m}\left(\frac{m+1}{s+1}\right)^{s+1}\\
&=\frac{s^s}{m}\left(\frac{\underbrace{\frac{m}{s}+\cdots+\frac{m}{s}}_s+1}{s+1}\right)^{s+1}\\[1mm]
&\geq\frac{s^s}{m}\left(\sqrt[s+1]{\frac{m}{s}\cdot\cdots\cdot\frac{m}{s}\cdot1}\right)^{s+1}\\
&=\frac{s^s}{m}\left(\frac{m}{s}\right)^{s}\\
&= m^{s-1} \ge 1,
\end{align*}
where the first inequality follows from the AM-GM inequality
$\frac{a_1+\cdots+a_t}{t}\geq\sqrt[t]{a_1\cdots a_t}$ and we are
done.
\end{proof}

Now let $o$ be a reference vertex in a hypergraph $H$ with vertex-set $V$, and define $x \in \mathbb{R}^V$ by
setting $x_v = g(\dist(o,v))$. Using the vector $x$, we get an
analog of nonnegative matrix inequality $Ax\geq \alpha x$, which
enables  us to compare the spectral radii of hypergraphs.

\begin{lemma}\label{lem_vectororder}
Let $H$ be a connected $k$-regular and $t$-uniform hypergraph with
adjacency tensor $\mathcal A$, and $x \in \mathbb{R}^V$ be defined as above.
Then $\mathcal{A} x \ge \varrho\,x^{[t-1]}$, where
$\varrho = \frac{t}{t-1}((t-1)(k-1))^{1/t}$.
\end{lemma}

\begin{proof}
We have $(\mathcal A x)_o = kg(1)^{t-1} = k\left(\frac{(t(1-\frac{1}{k}))^{1/(t-1)}}{((t-1)(k-1))^{1/t}}\right)^{t-1} = \frac{t}{t-1}((t-1)(k-1))^{1/t} = \varrho\,x_o \allowbreak = \allowbreak \varrho\,x_o^{t-1}$. Let $v\in V(H)$ with $\dist(o,v)=n\geq1$. Then $v$ is
adjacent to at least one vertex at distance $n-1$ from $o$.
Let $f$ be the corresponding edge. Since $g$ is non-increasing, we have:
\begin{align*}
 (\mathcal A x)_v &= \sum_{e\in E_v}x^{e\setminus\{v\}} =
  x^{f\setminus\{v\}} + \sum_{e\in E_v\setminus \{f\}}x^{e\setminus\{v\}}\\
 &\geq g(n-1)g(n)^{t-2}+(k-1)g(n+1)^{t-1}\\
 &=\frac{1+((t(1-\frac{1}{k}))^{1/(t-1)}-1)(n-1)}{((t-1)(k-1))^{\frac{n-1}{t}}}\left( \frac{1+((t(1-\frac{1}{k}))^{1/(t-1)}-1)n}{((t-1)(k-1))^{\frac{n}{t}}}\right)^{t-2}\\
   &\qquad+(k-1)\left(\frac{1+((t(1-\frac{1}{k}))^{1/(t-1)}-1)(n+1)}{((t-1)(k-1))^{\frac{n+1}{t}}}\right)^{t-1}\\
   &=\frac{t}{t-1}\,\frac{1}{((t-1)(k-1))^{\frac{n(t-1)-1}{t}}}\, \Bigl((1-\tfrac{1}{t})\hat{g}(n-1)\hat{g}(n)^{t-2}+\tfrac{1}{t}\hat{g}(n+1)^{t-1}\Bigr).
\end{align*}
Since
\begin{align*}
 \hat{g}(n+1)^{t-1} &= \left[1+((t(1-\tfrac{1}{k}))^{1/(t-1)}-1)(n+1)\right]^{t-1}\\
 &=\left[\hat{g}(n)+(t(1-\tfrac{1}{k}))^{1/(t-1)}-1)\right]^{t-1}\\
 &\geq \hat{g}(n)^{t-1}+(t-1)\hat{g}(n)^{t-2}\left[(t(1-\tfrac{1}{k}))^{1/(t-1)}-1\right],
\end{align*}
the above inequality continues as
\begin{align*}
 &\frac{t}{t-1}\,\frac{1}{((t-1)(k-1))^{\frac{n(t-1)-1}{t}}}\, \Bigl[(1-\tfrac{1}{t})\hat{g}(n-1)\hat{g}(n)^{t-2}+\tfrac{1}{t}\hat{g}(n+1)^{t-1}\Bigr]\\
 &\geq \frac{t}{t-1}\,\frac{1}{((t-1)(k-1))^{\frac{n(t-1)-1}{t}}}\, \Bigl[(1-\tfrac{1}{t})\hat{g}(n-1)\hat{g}(n)^{t-2}\\
 &+\tfrac{1}{t}(\hat{g}(n)^{t-1}+(t-1)\hat{g}(n)^{t-2}(t(1-\tfrac{1}{k}))^{1/(t-1)}-1))\Bigr].
\end{align*}
Since
\begin{align*}
   &(1-\tfrac{1}{t})\hat{g}(n-1)\hat{g}(n)^{t-2} + \tfrac{1}{t}\left[\hat{g}(n)^{t-1}+(t-1)\hat{g}(n)^{t-2}(t(1-\tfrac{1}{k}))^{1/(t-1)}-1)\right]\\
   &=\tfrac{1}{t}\hat{g}(n)^{t-1}+(1-\tfrac{1}{t})\hat{g}(n)^{t-2}\left[\hat{g}(n-1)+(t(1-\tfrac{1}{k}))^{1/(t-1)}-1)\right]\\
   &=\tfrac{1}{t}\hat{g}(n)^{t-1}+(1-\tfrac{1}{t})\hat{g}(n)^{t-2}\hat{g}(n)\\
   &=\hat{g}(n)^{t-1},
\end{align*}
the original inequality continues as
\begin{align*}
 &\frac{t}{t-1}\,\frac{1}{((t-1)(k-1))^{\frac{n(t-1)-1}{t}}}\,\hat{g}(n)^{t-1}\\
&=\frac{t}{t-1}((t-1)(k-1))^{1/t}\frac{\hat{g}(n)^{t-1}}{((t-1)(k-1))^{\frac{n(t-1)}{t}}}\\
 &=\frac{t}{t-1}((t-1)(k-1))^{1/t}g(n)^{t-1}.
\end{align*}
Therefore we have $(\mathcal A x)_v\geq
\frac{t}{t-1}((t-1)(k-1))^{1/t}g(n)^{t-1}=\varrho\, x_v^{t-1}$.
\end{proof}

By Lemma~\ref{lem_vectororder}, we have $\mathcal A x\geq\varrho
x^{[t-1]}$ and so $x^T(\mathcal Ax)\geq\varrho x^T x^{[t-1]} =
\varrho\|x\|_t^t$ if the hypergraph is finite, thus $\rho(\mathcal
A)\geq\frac{x^T(\mathcal A x)}{\|x\|_t^t}=\varrho$ in this case.
But, when the hypergraph is infinite, it is not immediately clear that
$\|x\|_t < \infty$. Fortunately,
we can say so whenever the hypergraph is finite or infinite by what
we get in the following.

\begin{theorem}\label{thm_radius}
Let $H$ be a $k$-regular, $t$-uniform hypergraph (finite or infinite). Then
$$\rho(H)\geq\frac{t}{t-1}((t-1)(k-1))^{1/t}.$$
\end{theorem}

\begin{proof}
We may assume $H$ is connected, and further it suffices to
consider the infinite case by the arguments given earlier. In the
hypergraph $H$, let $B_n$ denote the set of vertices at distance
at most $n$ from the reference vertex $o$, and $S_n=B_n\setminus
B_{n-1}$. Let $x\in \mathbb{R}^V$ be the function used in Lemma
\ref{lem_vectororder}, i.e., $x_v=g(\dist(v,o))$ for $v\in V$.
Define a function $x_n: V(H)\rightarrow \mathbb{R}$ by
$x_n(v)=x_v$ if $v\in B_n$ and $x_n(v)=0$ otherwise. Then for any
edge $e\in E(H)$, we have $x_n^e=x^e$ if $e\subseteq B_n$, and
$x_n^e=0$ otherwise. Thus, $x_n^{T}(\mathcal A x_n) = t\sum_{e\in
E(H)} x_n^e = t\sum_{e\in E(B_n)}x^e$. Denote by $E[S_n, S_{n+1}]$
the set of edges whose vertices have a  partition such that one
part is in  $S_n$ and the other in $S_{n+1}$. By using Lemma
\ref{lem_vectororder}, we obtain
\begin{align}
 x_n^{T}(\mathcal A x_n) &= t\sum_{e\cap B_n\neq\emptyset}x^e - ~t\sum_{e\in E[S_n, S_{n+1}]}x^{e} \nonumber\\
 &=\sum_{v\in B_n}x_v\sum_{e\in E_v}x^{e\setminus\{v\}} - ~t\sum_{e\in E[S_n, S_{n+1}]}x^{e} \nonumber\\
 &\geq \sum_{v\in B_n}x_v(\varrho x_v^{t-1}) - ~t\sum_{e\in E[S_n, S_{n+1}]}x^{e} \nonumber\\
 &= \varrho\|x_n\|_t^t - ~t\sum_{e\in E[S_n, S_{n+1}]}x^{e} \nonumber\\
 &\geq \varrho\|x_n\|_t^t-t(k-1)|S_n|\,g(n)^{t}.\label{eq:star1}
\end{align}
Note that $\|x_n\|_t^t=\sum_{i=0}^n|S_i|g(i)^t$. Dividing
by $\|x_n\|_t^t$ on both sides of the above inequality, we
have
\begin{equation}
 \rho(H) \geq \frac{x_n^{T}(\mathcal A x_n)}{\|x_n\|_t^t}
 \geq \varrho-t(k-1)\frac{|S_n|g(n)^{t}}{\sum_{i=0}^n|S_i|g(i)^t}.\label{eq:5a}
\end{equation}
For each $i$, $1\le i\leq n$, we have
\begin{equation}\label{e_sum}
|S_n|g(n)^{t}\leq
|S_i|((t-1)(k-1))^{n-i}g(n)^{t}=|S_i|g(i)^{t}\left(\frac{1+((t(1-\frac{1}{k}))^{1/(t-1)}-1)n}{1+((t(1-\frac{1}{k}))^{1/(t-1)}-1)i}\right)^{t}.
\end{equation}
Thus we have
\begin{align*}
    \frac{|S_n|g(n)^{t}}{\sum_{i=0}^n|S_i|g(i)^t}&=\frac{|S_n|g(n)^{t}}{1+\sum_{i=1}^n|S_i|g(i)^t}\\
&\leq  \frac{|S_n|g(n)^{t}}{\sum_{i=1}^n|S_i|g(i)^t}\\
     &\leq\frac{(1+((t(1-\frac{1}{k}))^{1/(t-1)}-1)n)^t}{\sum_{i=1}^n(1+((t(1-\frac{1}{k}))^{1/(t-1)}-1)i)^t}.
\end{align*}
Construct two sequences $\{a_n\}$ and $\{b_n\}$ with
$a_n=(1+((t(1-\frac{1}{k}))^{1/(t-1)}-1)n)^t$ and
$b_n=\sum_{i=1}^na_i$. It is easy to see that the sequence
$\{b_n\}$ is strictly increasing and approaches $+\infty$. Since
\begin{align*}
\frac{a_n-a_{n-1}}{b_n-b_{n-1}}=\frac{a_n-a_{n-1}}{a_n}=1-\frac{a_{n-1}}{a_n}
\end{align*}
and
\begin{align*}
\lim_{n\rightarrow \infty}\frac{a_{n-1}}{a_n}=\lim_{n\rightarrow
\infty}\left(\frac{1+((t(1-\frac{1}{k}))^{1/(t-1)}-1)(n-1)}{1+((t(1-\frac{1}{k}))^{1/(t-1)}-1)n}\right)^t=1,
\end{align*}
by Stolz-Ces\`aro theorem, we have
\begin{align*}
  \lim_{n\rightarrow \infty}\frac{a_n}{b_n}=\lim_{n\rightarrow \infty}\frac{a_n-a_{n-1}}{b_n-b_{n-1}}=1-\lim_{n\rightarrow
  \infty}\frac{a_{n-1}}{a_n}=1-1=0.
\end{align*}
This implies that the last term in (\ref{eq:5a}) converges to 0 as $n\to\infty$, and thus we conclude that $\rho(H)\ge\varrho$.
\end{proof}

We have the following corollary.

\begin{coro} \label{cor:acyclic}
Let $H$ be a (finite or infinite) acyclic hypergraph with maximum degree at most\/ $k$. Then
$$
   \rho(H) \le \frac{t}{t-1}((t-1)(k-1))^{1/t}.
$$
\end{coro}

\begin{proof}
Observe that $H$ is linear and thus a subhypergraph of the hypertree $T_{t,k}$. Therefore,
$$
   \rho(H) \le \rho(T_{t,k}) = \frac{t}{t-1}((t-1)(k-1))^{1/t}.
$$
\end{proof}

We are ready for the extension of the Alon-Boppana Theorem to hypergraphs. Of course, speaking of $\lambda_2(H)$ makes sense only in the case of finite hypergraphs.

\begin{theorem}
\label{thm_AlonBoppana hypergraphs}
For every finite $k$-regular, $t$-uniform hypergraph $H$ on $n$ vertices, we have
$$\lambda_2(H)\geq\frac{t}{t-1}((t-1)(k-1))^{1/t}-o_n(1).$$
\end{theorem}

\begin{proof}
Suppose the diameter of $H$ is $D$. Let $s$ be the smallest
positive nontrivial divisor of $t$, noting that $s>1$. Let
$d=\lfloor\frac{D}{2s-2}\rfloor-1$. We can select $s$
vertices $v_1, v_2, \ldots, v_s$  from a shortest path in $H$ joining two vertices at distance $D$ so
that $\dist(v_a, v_b)\ge 2d+2$ for any $1\le a < b \le s$. Classify the
vertices in $H$ according to their distance from $v_1, \ldots,v_s$ as follows:
\begin{align*}
    S_i^j &= \{w\in V(H) : \dist(w, v_j)=i\},\quad i=0,\ldots, d;~ j=1,\ldots, s;\\
    T     &= V(H)\setminus \cup_{0\leq i\leq d}\cup_{1\leq j\leq s} S_i^j.
\end{align*}

Let $\omega$ be the primitive $s^{\rm th}$ root of unity. Let
$U_j=\bigcup_{0\leq i\leq d} S_i^j$ for $j=1,\dots,s$. Define $s+1$
mappings $y_1,\ldots, y_s, y: V(H)\rightarrow \mathbb{C}$ as follows:
 \[
 y_j(u)=\left\{
   \begin{array}{cl}
     g(i), & \hbox{$u\in S_i^j,\ i=0,1,\ldots, d$;} \\
     0, & \hbox{otherwise,}
   \end{array}
 \right.
 \]
and
 \[
 y(u)=\left\{
   \begin{array}{cl}
     c_j\omega^{j-1}g(i), & \hbox{$u\in S_i^j,~ i=0,1,\ldots, d;~ j=1,\ldots, s$;} \\
     0, & \hbox{otherwise,}
   \end{array}
 \right.
 \]
where $c_1, c_2, \ldots, c_s$ are positive constants so that
$$c_1\sum_{u\in U_1} y_1(u) = c_2\sum_{u\in U_2} y_2(u) = \cdots = c_s\sum_{u\in U_s} y_s(u) = 1.$$
By this assumption, we have that
\begin{align*}
 \sum_{u\in V(H)}y(u)&=\sum_{u\in U_1}y(u)+\sum_{u\in U_2}y(u)+\cdots+\sum_{u\in U_s}y(u)+\sum_{u\in T}y(u)\\
&=c_1\sum_{u\in U_1}y_1(u)+c_2\omega\sum_{u\in U_2}y_2(u)+\cdots+c_s\omega^{s-1}\sum_{u\in U_s}y_s(u)+0\\
&=1+\omega+\cdots+\omega^{s-1}\\
&=0
\end{align*}
and it follows immediately that
$$y^T(Jy) = \sum_{v_1,v_2,\ldots,v_t}y_{v_1}y_{v_2}\cdots y_{v_t} = \Bigl(\sum_v y_v\Bigr)^t=0.$$

Observe that $y$ is a constant multiple of $y_j$ on each $U_j$ ($1\le j\le s$) and that $y_j$ is defined in the same way as the vector $x_n$ (with $n=d$) in the proof of Theorem \ref{thm_radius}. Moreover, vectors $y_1,\dots, y_s$ have disjoint supports $U_1,\dots,U_s$, and no vertex in $U_a$ is adjacent to any vertex in $U_b$ whenever $1\le a<b\le s$. The inequality (\ref{eq:star1}) from the proof of Theorem \ref{thm_radius} thus holds for each $y_j$. This implies the following:
\begin{align*}
  y^T(\mathcal Ay) &= \sum_{j=1}^s c_j^t \omega^{t(j-1)} y_j^T(\mathcal A y_j) \\
  &= \sum_{j=1}^s c_j^t y_j^T(\mathcal A y_j)\\
  &\geq \sum_{j=1}^s c_j^t (\varrho\|y_j\|_t^t - t(k-1)|S_d^j|\,g(d)^{t})\\
  &= \varrho\|y\|_t^t - \sum_{j=1}^s c_j^t t|S_d^j|(k-1)g(d)^{t},
\end{align*}
where $S_d^j$ denotes the set of vertices at distance $d$ from $v_j$ ($1\le j\le s$).
Thus
 \begin{align*}
   \lambda_2(H)&=\|\mathcal A - \tfrac{t|E(H)|}{n^t}J\|\\
&\geq \frac{y^T((\mathcal A-\frac{t|E(H)|}{n^t}J)y)}{\|y\|_t^t}\\
&=\frac{y^T(\mathcal Ay)}{\|y\|_t^t}\\
&\geq\varrho-t(k-1)\frac{\sum_{j=1}^sc_j^t|S_d^j|g(d)^{t}}{\sum_{j=1}^sc_j^t\sum_{i=0}^d|S_i^j|g(i)^t}.
 \end{align*}
As in the proof of Theorem~\ref{thm_radius}, we see that
$\frac{|S_d^j|g(d)^{t}}{\sum_{i=0}^d|S_i^j|g(i)^t}\rightarrow0$
with $d\rightarrow\infty$ for every $j=1,\ldots, s$. Let
$a_j:=|S_d^j|g(d)^{t}$ and $b_j=\sum_{i=0}^d|S_i^j|g(i)^t$ for
$j\in[s]$. By the inequality that
$\min\frac{a_j}{b_j}\leq\frac{a_1+a_2+\cdots+a_s}{b_1+b_2+\cdots+b_s}\leq\max\frac{a_j}{b_j}$,
we conclude that
$$t(k-1)\frac{\sum_{j=1}^sc_j^t|S_d^j|g(d)^{t}}{\sum_{j=1}^sc_j^t\sum_{i=0}^d|S_i^j|g(i)^t} = t(k-1)\frac{\sum_{j=1}^sc_j^ta_j}{\sum_{j=1}^sc_j^tb_j}\rightarrow0 \quad \hbox{as } d\rightarrow\infty.$$
As $n\rightarrow\infty$, we have
$D\rightarrow\infty$ and so $d = \lfloor\frac{D}{2s-2}\rfloor-1\rightarrow\infty$. This completes the proof.
\end{proof}

Serre \cite{Serre} gave a strengthening of the Alon-Boppana Theorem showing that a positive proportion of eigenvalues
of any $k$-regular graph must be bigger than $2\sqrt{k-1}-o(1)$. The above proof can be used to give the same kind of an extension to hypergraphs, except that it is not clear what would be the meaning of the third (fourth, fifth, etc.) largest eigenvalue of a hypergraph. For our purpose, the following definition seems to be generous enough.

Two vectors $x,y\in \mathbb{C}^V$ are \emph{strongly orthogonal\/} for a tensor $\mathcal A$ of order $t$ if the vectors $\mathcal{A}^p x$ and $\mathcal{A}^q y$ are orthogonal for every $p,q\in \{0,1,\dots,t\}$. For a positive integer $j$, let $\mathcal{X}_j$ be the set of all $\{x_1,\dots,x_j\}$ of $j$ complex vectors such that $\Vert x_l\Vert_t = 1$ for $l=1,\dots,j$ and $x_l$ and $x_m$ are strongly orthogonal for every $1\le l < m \le j$.
Then we define the $j$th multilinear value for $\mathcal A$ as
$$
  \mu_j(\mathcal{A}) = \sup_{X\in\mathcal{X}_j} \inf_{x\in X} x^T(\mathcal{A}x).
$$
(If $\mathcal{X}_j$ is empty, then we set $\mu_j(\mathcal{A}) = -\infty$.)
It can be shown that $\mu_1(\mathcal{A})$ and $\mu_2(\mathcal{A})$ correspond to $\lambda_1(H)$ and $\lambda_2(H)$ if $\mathcal A$ is the adjacency tensor for a regular $t$-uniform hypergraph $H$.

As mentioned above, the proof of Theorem \ref{thm_AlonBoppana hypergraphs} can be adapted to yield the following extension.

\begin{theorem}
\label{thm_Serre hypergraphs}
For every positive integer $j$ and every finite $k$-regular, $t$-uniform hypergraph $H$ on $n$ vertices, we have
$$\mu_j(H) \geq \frac{t}{t-1}((t-1)(k-1))^{1/t}-o_n(1).$$
\end{theorem}

\begin{proof}
The proof is essentially the same as the proof of Theorem \ref{thm_AlonBoppana hypergraphs} with some added details. Thus we only give a sketch of the proof. First, we need sufficiently large diameter so that $j$ vectors $x_1,\dots, x_j$ can be found, whose form is as given in the proof of Theorem \ref{thm_AlonBoppana hypergraphs} and whose supports\footnote{The \emph{support} of a vector $x\in \mathbb{C}^V$ is the set of vertices $v$ for which $x_v\ne 0$.} are mutually at distance more than $2t$ from each other so that the supports of $\mathcal{A}^p x_l$ and $\mathcal{A}^q x_m$ are disjoint for every $l\ne m$ and every $p,q\in \{0,1,\dots,t\}$. Using the notation from the previous proof, we start by selecting $sj$
vertices $v_m^l$ ($m=1,\dots,s$ and $l=1,\dots,j$), mutually at distance at least $d+2t+1$ from each other. We use vertices $v_1^l,\dots,v_s^l$ to define the radial vectors $x_l$ ($l=1,\dots,j$) and use them to obtain a lower bound on $\mu_j(\mathcal{A})$. The details are left to the reader.
\end{proof}

The conclusion of Theorem \ref{thm_Serre hypergraphs} holds also when $j=j(n)$ increases with $n$. In fact, for every $\varepsilon>0$, there is a positive number $f(\varepsilon)$ such that $j$ can be as large as $\lfloor f(\varepsilon)n\rfloor$ when we want $\mu_j(H) \geq \frac{t}{t-1}((t-1)(k-1))^{1/t}-\varepsilon$.

\end{document}